\documentclass{amsart}
\usepackage[T1]{fontenc}
\usepackage[utf8]{inputenc}
\usepackage{eucal}
\usepackage[unicode]{hyperref}
\usepackage[capitalise]{cleveref}
\usepackage{xcolor}
\hypersetup{hypertexnames = false,
   bookmarksdepth = 2, bookmarksopen = true, bookmarksnumbered, colorlinks,
  linkcolor={red!50!black},
 citecolor={blue!50!black},
 urlcolor={blue!80!black},
 pdfstartview=}

\usepackage{amsfonts,amscd,amsmath,amssymb}
\usepackage{amsthm}
\usepackage{booktabs}
\usepackage{longtable}
\usepackage{parskip}
\usepackage{rotating}
\usepackage{xparse}
\usepackage{xspace}

\makeatletter\@addtoreset{equation}{section}\makeatother

\makeatletter\@addtoreset{subsection}{equation}\makeatother

\makeatletter
\newcommand{\symbitem}[1]{\item[#1]%
\renewcommand{\@currentlabel}{#1}\ignorespaces}
\makeatother

\newcounter{todocounter}
\DeclareDocumentCommand\addreference{g}{\stepcounter{todocounter}\todo[color
= blue!30]{\thetodocounter. Add reference\IfNoValueF{#1}{:
#1}}\xspace}
\DeclareDocumentCommand\checkthis{g}{\stepcounter{todocounter}\todo[color
= red!50]{\thetodocounter. Check this\IfNoValueF{#1}{:
#1}}\xspace}
\DeclareDocumentCommand\fixthis{g}{\stepcounter{todocounter}\todo[color
= orange!50]{\thetodocounter. Fix this\IfNoValueF{#1}{:
#1}}\xspace}
\DeclareDocumentCommand\expand{g}{\stepcounter{todocounter}\todo[color
= green!50]{\thetodocounter. Expand\IfNoValueF{#1}{: #1}}\xspace}

\newtheorem{theorem}[equation]{Theorem}
\newtheorem{lemma}[equation]{Lemma}

\newtheorem{corollary}[equation]{Corollary}
\newtheorem{proposition}[equation]{Proposition}

\newtheorem*{theorem*}{Theorem}
\newtheorem*{lemma*}{Lemma}
\newtheorem*{corollary*}{Corollary}
\newtheorem*{proposition*}{Proposition}
\theoremstyle{remark}
\newtheorem{remark}[equation]{Remark}
\newtheorem*{remark*}{Remark}
\theoremstyle{definition}

\newtheorem{example}[equation]{Example}

\makeatletter
\def\gitfootnote{\gdef\@thefnmark{}\@footnotetext}
\makeatother

\mathchardef\mhyphen="2D

\def \CC {\mathcal{C}}  
\def \OO {\mathcal{O}}  
\def \TT {\mathcal{T}}  
\def \VV {\mathcal{V}}      
\def \VVo {\overline{\VV}}  

\def \B {\mathbb{B}}    
\def \C {\mathbb{C}}    
\def \P {\mathbb{P}}    
\def \Q {\mathbb{Q}}    
\def \R {\mathbb{R}}    
\def \Z {\mathbb{Z}}    

\newcommand{\lra}{\longrightarrow}

\DeclareMathOperator{\Coh}{Coh} \DeclareMathOperator{\Hom}{Hom}
\DeclareMathOperator{\Pic}{Pic} \DeclareMathOperator{\Spec}{Spec}
\DeclareMathOperator{\Tr}{Tr}

\providecommand{\arxiv}[1]{\href{http://arxiv.org/abs/#1}{arXiv:#1}}
\providecommand{\doi}[1]{\href{http://dx.doi.org/#1}{\tt doi:#1}}

\begin{document}
\title{On automorphic forms of small weight for fake projective planes}
\author{Sergey Galkin}
\address{PUC-Rio, Departamento de Matem\'atica \\
Rua Marqu\^es de S\~ao Vicente 225, G\'avea, Rio de Janeiro}
\address{HSE University}
\address{Independent University of Moscow}
\email{hirzebruch-3594@galkin.org.ru}
\author{Ilya Karzhemanov}
\address{Laboratory of AGHA, Moscow Institute of Physics and Technology \\
9 Institutskiy per., Dolgoprudny, Moscow Region, 141701, Russia}
\email{karzhemanov.iv@mipt.ru}
\author{Evgeny Shinder}
\address{School of Mathematics and Statistics \\
University of Sheffield \\
The Hicks Building, Hounsfield Road, Sheffield, S3 7RH, United
Kingdom}
\address{HSE University}
\email{e.shinder@sheffield.ac.uk}


\thanks{
S.\,G. was supported in part by the Simons Foundation (Simons-IUM
fellowship), the Dynasty Foundation and RFBR (research project
~15-51-50045-a). The article was prepared within the framework of
a subsidy granted to the HSE by the Government of the Russian
Federation for the implementation of the Global Competitiveness
Program (RSF grant, project 14-21-00053 dated 11.08.14). I.\,K.
was supported by World Premier International Research Initiative
(WPI), MEXT, Japan, Grant-in-Aid for Scientific Research
(26887009) from Japan Mathematical Society (Kakenhi), and by the
Russian Academic Excellence Project ``5 -- 100''. We would like to
thank CIRM, Universit\`a degli Studi di Trento, and Kavli IPMU for
hospitality and excellent working conditions. Main body of this
work was done in August 2014, as a part of our ``Research in
Pairs'' at CIRM, Trento (Italy). We are also grateful to the
anonymous referee for valuable remarks.}

\date{\today}

\begin{abstract}
On the projective plane there is a unique cubic root of the
canonical bundle and this root is acyclic. On fake projective
planes such root exists and is unique if there are no 3-torsion
divisors (and usually exists, but not unique, otherwise). Earlier
we conjectured that any such cubic root must be acyclic. In the
present note we give two short proofs of this statement and show
acyclicity of some other line bundles on the fake projective
planes with at least $9$ automorphisms. Similarly to our earlier
work we employ simple representation theory for non-abelian finite
groups. The first proof is based on the observation that if some
line bundle is non-linearizable with respect to a finite abelian
group, then it should be linearized by a finite,
\emph{non-abelian}, Heisenberg group. For the second proof, we
also demonstrate vanishing of odd Betti numbers for a class of
abelian covers, and use linearization of an auxiliary line bundle
as well.
\end{abstract}

\maketitle


\setcounter{tocdepth}{1} \tableofcontents

\section*{Introduction}
\label{section:int}

\subsection{An overview}
\label{subsection:int-1} In this work we take another step towards
the categorification of \emph{Hirzebruch Proportionality
Principle} \cite{Hirzebruch58,Ise64,Hirzebruch66,Kobayashi-Ono}.
Let $X$ be a compact complex manifold uniformized by a bounded
Hermitian symmetric domain $D_- := K\backslash G_-$, so that $X =
D_-\slash\Gamma = K\backslash G_-/\Gamma$, where $\Gamma \simeq
\pi_1(X)$ and $K$ is the maximal compact subgroup in a Lie group
$G_-$. Let $G_+$ be a compact group that shares the same
complexification $G_c$ with $G_-$. Consider a compact Hermitian
symmetric space $F := D_+ = K\backslash G_+$ dual to $D_-$ in the
sense of E.\,Cartan. A remarkable fact, discovered by F.
Hirzebruch \cite{Hirzebruch58} around 1956, says that the classes
of $X$ and $F$ are proportional in Thom's bordism ring, with
proportionality constant being the Todd genus $a :=
\chi(X,\OO_X)$. In particular, $X$ and $F$ have the same slope
$\nu := \frac{\int c_2 c_1^{d-2}}{\int c_1^d}$, with non-zero
denominator because $K_X$ and $-K_F$ are ample. For example, all
compact manifolds uniformizable by the $d$-dimensional complex
ball $\B^d := \{ (z_1,\dots,z_d)\in\C^d\ \big\vert\ \sum |z_i|^2 <
1 \}$ are proportional to the complex projective space $\P^d$ in
this sense, and have slope $\nu = \frac{d}{2(d+1)}$ \footnote{~On
the other hand, for any $d$-dimensional compact $X$ with ample
$K_X$ S.-T. Yau demonstrated in \cite{Yau77} (as a corollary from
the Aubin -- Yau construction of the K\"ahler -- Einstein metric)
that the slope of $X$ is bounded by the slope of $\P^d$, and the
equality holds iff $X$ is uniformizable by $\B^d$.}, whereas all
compact manifolds uniformizable by the poly-disc $(\B^1)^d$ are
proportional to the Segre variety $(\P^1)^d$ and have slope
$\nu=\frac{1}{2}$. From his Proportionality Principle and the
Hirzebruch -- Riemann -- Roch theorem Hirzebruch deduced that
$\chi(X,\omega_X^{\otimes w}) = a\cdot\chi(F,\omega_F^{\otimes
w})$ for all integers $w$. \footnote{~Note that for
$X=\B^d/\Gamma$ this gives $\chi(X,\omega_X^{\otimes w}) =
\chi(X,\OO_X) \binom{d-w(d+1)}{d}$.} One computes the Euler
characteristic $\chi(X,L)$ of any line bundle $L\in\Pic X$,
numerically proportional to $K_X$ (that is $L^{\otimes q} \simeq
\omega_X^{\otimes p}$ for $w := \frac{p}{q}\in\Q$), in a similar
way.
Let us say then that a global section $0\neq s \in H^0(X,L)$ is a
\emph{modular form of large weight} if $w>1$, a \emph{modular form
of canonical weight} if $w=1$, and a \emph{modular form of small
weight} (or \emph{theta-characteristic}) if $0<w<1$. \footnote{~To
keep our notation compatible with the classical one-dimensional
case one could also say that $s\in H^0(X,L)$ has weight $w(d+1)
\in \Z$.} For $w > 1$, the Euler characteristic coincides with
$h^0(X,\omega_X^{\otimes w})$, thanks to the Kodaira vanishing.
For theta-characteristics, the Euler characteristic vanishes and
for the canonical weight it equals $1$, but there is no \emph{a
priori} vanishing and we know very little about $h^k(X,L)$ in the
dimension $d>1$. In \cite{Ise64}, M. Ise generalized Hirzebruch's
Proportionality to vector bundles of higher rank: to every
representation $\rho : K\to GL(r,\C)$ he associated a vector
bundle $E_X^\rho$ on $X$ and a homogeneous vector bundle
$E_F^\rho$ on $F$ (both of rank $r$) such that $\chi(X,E_X^\rho) =
a\cdot\chi(F,E_F^\rho)$, where the number $\chi(F,E^\rho)$ can be
computed via Borel -- Weil -- Bott theorem. R. Langlands has
independently obtained these formulae by different method (see
\cite{Langlands63}).

The preceding discussion was summarized in \cite{Hirzebruch66}
(cf. \cite[Section \textbf{17}]{Hirzebruch66}) as follows:

\begin{quote}
\emph{Given a cohomology theory $H^{\bullet}$ over $\Q$, admitting
the Poincar\'e duality and K\"unneth formula, there exists an
injective ring homomorphism $H^{\bullet}(F,\Q) \to
H^{\bullet}(X,\Q)$. Moreover, this homomorphism is given by an
element $\mathcal{K} \in H^{\bullet}(X \times F)$ via
$p_{1*}p_2^*\mathcal{K}$ for the projections $p_i$ onto $X$ and
$F$, respectively.}
\end{quote}

In \cite{Mumford79}, D. Mumford gave an ingenious construction of
\emph{a smooth complex projective surface with $K$ ample, $K^2=9$,
$p_g=q=0$}. All such surfaces are now known under the name of
\emph{fake projective planes}. They have been recently classified
into $100$ isomorphism classes by G. Prasad, S.-K. Yeung
\cite{Prasad-Yeung} and D. Cartwright, T. Steger
\cite{Cartwright-Steger}. According to \cite{Yau77} universal
cover of any fake projective plane is the complex ball and the
papers \cite{Prasad-Yeung,Cartwright-Steger,CScode} explicitly
describe all subgroups in the automorphism group of the ball which
are the fundamental groups of fake projective planes (cf.
\cite{Ishida88} and \cite{Kato-Ochiai}). However, these surfaces
are poorly understood from the algebro-geometric perspective,
since the uniformization maps (both complex and, as in the
Mumford's case, $2$-adic) are highly transcendental. Most notably,
the S. Bloch's conjecture on zero-cycles \cite{Bloch80} is not
established yet for \emph{any} fake projective plane, and we refer
to \cite{Barlow} for a compatible result for the Mumford's
surface. \footnote{~Recently L. Borisov and J. Keum
\cite{Borisov-Keum} have given the first explicit
\emph{algebro-geometric} construction of a fake projective plane
(see also \cite{Borisov-Fatighenti} and \cite{BBF}).}

Earlier we have initiated the study of fake projective planes from
the homological algebra perspective (see \cite{GKMS-mp,GKMS}).
Namely, for $\P^2$ the corresponding bounded derived category of
coherent sheaves has a semi-orthogonal decomposition, $D^b(\P^2) =
\left<\OO, \OO(1), \OO(2)\right>$, as was shown by A. Beilinson in
\cite{Beilinson78}. The ``easy'' part of his argument was in
checking that the line bundles $\OO(1)$ and $\OO(2)$ are
acyclic,\footnote{~To distinguish the class of acyclic objects
\emph{with vanishing global sections} \cite{ABKW} proposed the
term \emph{immaculate}.} and also that any line bundle on $\P^2$
is exceptional. All these results follow from Serre's computation.
In turn, the ``hard'' part consisted of checking that $\OO$,
$\OO(1)$, $\OO(2)$ actually generate $D^b(\P^2)$, which is
equivalent to the existence of \emph{Beilinson's spectral
sequence} --- resolution for the structure sheaf of the diagonal
in $\P^2\times\P^2$ with terms of the form $A \boxtimes B$ (hence,
in particular, it corresponds to an object in $D^b(\P^2 \times
\P^2)$). However, as follows from \cite{GKMS}, for fake projective
planes one can not construct a full exceptional collection this
way. But still one can define analogues of $\OO(1),\OO(2)$ for
some of these surfaces (see Section~\ref{section:pre} below) and
try to establish the ``easy part'' for them. Then exceptionality
of line bundles is equivalent to the vanishing of $h^{0,1}$ and
$h^{0,2}$, which is clear due to $q = p_g = 0$, while
acyclicity/immaculacy is not at all obvious.

We are going to treat acyclicity/immaculacy problem, more
generally, in the context of ball quotients and modular forms. Our
argument proceeds by proving the absence of modular forms of small
weights on complex balls (compare with \cite{Ise64, Weissauer83}).
Some categorical aspects of this approach are discussed in
Remark~\ref{proportionality} at the end of this section.
\refstepcounter{equation}
\subsection{Structure of the paper}
\label{subsection:int-2}

In Section~\ref{section:betti}, we prove Lemma~\ref{lemma:betti}
about vanishing of some Betti numbers $b_{2l+1}(X)$ for a class of
compact K\"ahler manifolds $X$, which includes the abelian covers
of fake projective spaces, minifolds (see \cite{GKMS-mp}), and
many other examples.

Further, in Section~\ref{section:weight3} we show how the
vanishing $b_1(X) = 0$ implies regularity of the universal abelian
covers of fake projective planes, and thus we deduce the
uniqueness of modular forms of canonical weight in
Theorem~\ref{theorem:weight3}. This result is \emph{a priori}: its
proof only uses the definition of fake projective planes but not
their classification.

The rest of the paper studies fake projective planes $S$ with a
faithful action of the automorphism group $A_S$ of order $\ge 9$.

Section~\ref{section:pre} sets up the equivariant scene. We recall
the construction of tautological line bundle $\OO_S(1)$ on $S$ and
discuss how its existence and linearizability (with respect to the
automorphism group) are related to the liftings of the fundamental
group of $S$ to $SU(2,1)$. We also show that a non-linearizable
faithful action of the abelian group $(\Z/3)^2$ gives rise to a
linearizable non-faithful action of its central extension $G$ ---
the \emph{Weyl -- Heisenberg group} $H_3$. Finally, we recall the
holomorphic Lefschetz-type formula of Atiyah -- Bott, which is an
indispensable tool for explicit computations with equivariant
sheaves.

In Section~\ref{section:weight-large}, we compute the characters
of spaces of modular forms of large weight considered as
$G$-representations (cf. Section {\ref{subsection:pre-1-a}}
below), and in Section~\ref{section:weight2} we deduce the absence
of modular forms of small weight from the computations for large
weight.

\refstepcounter{equation}
\subsection{Methods, arguments, and relation with other works}
\label{subsection:int-3}

The proof of Lemma~\ref{lemma:betti} recycles a trick from the
proofs of \cite[Lemma A.4]{Galkin-Iritani}, \cite[Lemma
4]{Stover13}, \cite[Proposition 1.2]{HMT}, \cite[Lemma
2.14]{Klingler03} (note that it can be generalized to actions of
abelian groups on Hodge -- Lefschetz -- Frobenius
super-commutative algebras).

Our lemma is complementary to various vanishing results and has a
peculiar domain of applicability that includes some ramified and
unramified abelian covers. One can obtain even stronger results
(by using variants of Lefschetz hyperplane theorem) in the case of
\emph{totally ramified} cyclic covers. For example, applying
Nori's variant of weak Lefschetz theorem, V. Kharlamov and V.
Kulikov prove the following:

\begin{theorem*}[see {\cite[Proposition 1]{Kharlamov-Kulikov14}}]
\label{theorem-kk14} If $f: Y \longrightarrow X$ is a finite
cyclic cover of smooth complex surfaces, totally ramified over a
smooth irreducible curve $B\subset X$ with $B^2>0$, then the
induced homomorphism $f_*:\pi_1(Y)\to\pi_1(X)$ is an isomorphism.
\footnote{~Similar result extends to higher-dimensional $X$ and
$Y$.}
\end{theorem*}

Apart from Lemma~\ref{lemma:betti} we know of two other approaches
which in some cases prove the vanishing of $b_1$ for ball
quotients: these are the theorem of J. Rogawski
\cite{Rogawski90,Blasius-Rogawski} for congruence subgroups and D.
Kazhdan's property (T) \cite{Kazhdan67}, as used originally by D.
Mumford \cite{Mumford79}, M.-N. Ishida \cite{Ishida88}, and
subsequently by N. Fakhruddin \cite{Fakhruddin15}.

\begin{theorem*}[Rogawski's vanishing, {see \cite[Theorem 15.\,3.\,1]{Rogawski90}, \cite{Blasius-Rogawski}}]
\label{theorem-rog} Let $\Pi$ be a torsion-free cocompact
arithmetic subgroup in $PU(2,1)$ \emph{of type II} (see
\cite{Klingler03,Yeung04}). For $S = \B/\Pi$, if $\Pi$ is a
\emph{congruence subgroup}, then $b_1(S) = 0$ and the Picard
number of $S$ equals $1$.
\end{theorem*}

We sketched the proof of Lemma~\ref{lemma:betti} back in 2014
(although we did not use it in the first preprint version of this
article because we thought that Rogawski's vanishing would
suffice). We thank the anonymous person who has kindly pointed out
to us that it is not a priori clear whether the fundamental groups
of the abelian covers (or of fake projective planes themselves)
which we consider, are congruence subgroups (in fact it is likely
that most of these fundamental groups are \emph{not} congruence).

In turn, the use of Kazhdan's property (T) is even more intricate
--- the fundamental group $\pi_1(S) \subset PU(2,1)$ itself never
has this property, however for $2$-adically uniformized fake
projective planes abelianization $H_1(S,\Z) =
\pi_1(S)/[\pi_1(S),\pi_1(S)]$ is isomorphic to the abelianization
$\Gamma_2/[\Gamma_2,\Gamma_2]$ of a $2$-adic lattice $\Gamma_2
\subset PGL_3(\Q_2)$, which satisfies property (T). In the
approaches of \cite{Mumford79,Ishida88,Fakhruddin15}, Kazhdan's
property \emph{still has to be used}, for there are no alternative
ways to see that the $2$-adic construction produces surfaces with
$b_2=1$.

In particular, Fakhruddin considers fake projective planes which
admit $2$-adic uniformization with a torsion-free covering group,
and after the formulation of \cite[Proposition 3.1]{Fakhruddin15}
he writes:

\begin{quote}
\emph{It seems reasonable to expect that a similar result holds
for all line bundles on all fake projective planes.}
\end{quote}

Our Lemma~\ref{lemma:regularity} and Theorem~\ref{theorem:weight3}
confirm that the claim (2) in loc. cit. is indeed \emph{true for
all fake projective planes}. We also note that part (3) there is
obvious, whereas part (1) corresponds to the wishful vanishing of
theta-characteristics, which is discussed in the second part of
our paper. Here the linearization of auxiliary line bundles on $S$
enters the game.

Recall that all fake projective planes with at least $9$
automorphisms fall into the six cases represented in Table A below
(cf. \cite{CScode} and \cite[Section 6]{GKMS}). There one denotes
by $\Pi$ the fundamental group of $S$, so that $S = \B\slash\Pi$
for the unit ball $\B\subset\C\P^2$, and $N(\Pi)$ denotes the
normalizer of $\Pi$ in $PU(2,1)$.

One of the principle observations is that the group $\Pi$ lifts to
$SU(2,1)$. The lifting produces a line bundle $\OO_S(1)\in\Pic S$
such that $\OO_S(3) := \OO_S(1)^{\otimes 3} \simeq \omega_S$ ---
the canonical sheaf of $S$. Moreover, the preimage
$\widetilde{N(\Pi)} \subset SU(2,1)$ of $N(\Pi) \subset PU(2,1)$
acts fiberwise-linearly on the total space $\text{Tot}\,\OO_S(1)
\lra S$, which provides a natural linearization for $\OO_S(1)$
(and consequently for all other $\OO_S(k)$). Furthermore, the
action of the group $\Pi\subset \widetilde{N(\Pi)}$ is trivial, so
that the graded algebra $\displaystyle\bigoplus_{k\in\Z}
H^0(S,\OO_S(k))$ is endowed with the structure of a $G$-module,
where $G := \widetilde{N(\Pi)}\slash\Pi$. In the same way one
obtains the structure of a $G$-module on
$H^0(S,\OO_S(k)\otimes\varepsilon)$ for any choice of
$G$-linearization of an $A_S$-invariant torsion line bundle
$\varepsilon \in \Pic^0S$ (cf. Section {\ref{subsection:pre-1}}
below). For $k \geq 4$, the $G$-character of the action on
$H^0(S,\OO_S(k)\otimes\varepsilon)$ can be computed via Hirzebruch
proportionality, Kodaira vanishing and Atiyah -- Bott holomorphic
Lefschetz formula (this is done in
Section~\ref{section:weight-large}).

It turns out that for $k\in\{4,5\}$ and $G = G_{21}$ or $G = H_3$
the vector spaces $H^0(S,\OO_S(k)\otimes\varepsilon)$ are direct
sums of irreducible $3$-dimensional $G$-representations.
Furthermore, multiplication by $0 \ne s_\varepsilon \in
H^0(S,\omega_S \otimes \varepsilon)$ (cf.
Theorem~\ref{theorem:weight3}) is a monomorphism of
$G$-representations, and it is easy to compute that
$H^0(S,\OO_S(2)\otimes\varepsilon) = 0$ (see
Section~\ref{section:weight2}). Note that
Theorem~\ref{theorem:weight2} and the arguments in the proof of
\cite[Proposition 3.10]{Brino-Cerbo} imply that the linear system
$|2K_S|$ gives a bicanonical embedding of $S$ (cf. Theorem 3.8 in
loc. cit). It also complements the results in \cite[Section
4]{Brino-Cerbo}. Related results are obtained by different methods
in \cite{Fakhruddin15,Keum17,Lai-Yeung,Brino-Cerbo}.

\begin{remark}
\label{proportionality} When the metric on the Hermitian line
bundle $\OO_S(1)$ is concerned one may replace $S$ with its
universal cover $\B \subset \P^2$. Then $\OO_S(1)$ can be
identified (locally) with $\OO_{\P^2}(2) \otimes \omega_{\P^2}$
(cf. Section {\ref{subsection:pre-1}} below). In other words, the
global sections of $\OO_S(1)$ are those of $\OO_{\P^2}(2)$,
``twisted by a functional determinant'' (cf. \cite[Section {\bf
14}]{Hirzebruch66}), which suggests a natural identification
$H^0(S,\OO_S(1)) = H^2(\P^2,\OO_{\P^2}(-2)) \ (= 0)$. Similar
considerations apply to $\OO_S(2)$. Furthermore, since $\OO(\pm
1)$ are cubic roots of $K$ (one for $S$ and $\P^2$) and
$K\big\vert_{\mathbb{B}}$ admits a \emph{natural} trivialization
(holomorphic volume), which we fix, one may propose a
\emph{duality} $H^i(S,\OO_S(j)) \simeq
H^{i}(\P^2,\OO_{\P^2}(-j))^*$ to hold in general, where $j \in
\Z/3$, $i \ge 0$. (The trace map is obtained from
$H^{i}(\P^2,\OO_{\P^2}(-j)) \simeq H^{2-i}(\P^2,\OO_{\P^2}(j)
\otimes \omega_{\P^2})$, restriction (local) $(\OO_{\P^2}(j)
\otimes \omega_{\P^2})\big\vert_S = \OO_S(-j) \otimes \omega_S$,
and the fact that $H^2(S,\omega_S) \simeq \C$.) More to the point,
we expect that there exists an object $\mathcal{K} \in D^b(S
\times \P^2)$ (a ``Beilinson -- Green -- Penrose kernel''), which
implements the asserted duality (cf. the discussion in Section
{\ref{subsection:int-1}}). This categorical counterpart of the
Proportionality Principle will be developed elsewhere. It would be
interesting, though, to compare our present considerations with
the \emph{twistor correspondence} between (Green) functions on
differentiable $4$-manifolds and sheaves on the respective twistor
spaces, as in e.g. \cite{LeBrun04}.
\end{remark}

\begin{table}[h]
\begin{longtable}{|c|c|c|c|c||c|c|c|c|c|c|}
\hline $l$ or $\CC$ & $p$ & $\TT_1$ & $N$ & $\#\Pi$ &
$A_S$ & $H_1(S,\Z)$ & $\Pi$ lifts? & $H_1(S\slash A_S,\Z)$ & $N(\Pi)$ lifts? \\
\hline
$\Q(\sqrt{-7})$ & $2$ & $\emptyset$ & $21$ & $3$ & $G_{21}$ & $(\Z\slash2)^4$ & yes & $\Z\slash 2$ & yes \\
& & $\{7\}$ & $21$ & $4$ & $G_{21}$ & $(\Z\slash2)^3$ & yes & $0$ & yes  \\
\hline $\CC_{20}$ & $2$ & $\emptyset$ & $21$ & $1$ &
$G_{21}$ & $(\Z\slash2)^6$ & yes & $0$ & yes \\
\hline \hline $\CC_{2}$ & $2$ & $\emptyset$ & $9$ & $6$ &
$(\Z/3\Z)^2$ & $\Z\slash14$ & yes & $\Z\slash 2$ & no  \\
& & $\{3\}$ & $9$ & $1$ &
$(\Z/3\Z)^2$ & $\Z\slash7$ & yes & $0$ & no \\
\hline $\CC_{18}$ & $3$ & $\emptyset$ & $9$ & $1$ &
$(\Z/3\Z)^2$ & $\Z\slash26\times\Z\slash2$ & yes & $0$ & no \\
\hline
\end{longtable}
\text{Table A}
\end{table}

\section{Vanishing of odd Betti numbers for abelian covers}
\label{section:betti}

\begin{lemma}
\label{lemma:betti} Let $A:N$ be an action of a finite abelian
group $A$ on a compact K\"ahler manifold $N$ of dimension $n =
\frac{\dim_{\R} N}{2}$. If the $A$-invariant forms of type $(k,k)$
satisfy $\dim H^{k,k}(N)^A = 1$, for an odd $k\leq \frac{n}2$,
then $b_k(N) = b_{k-2}(N) = b_{k-4}(N) = ... = b_1(N) = 0$.
\end{lemma}

\begin{proof}
Recall that the space $H^{l,l}(N,\R) = H^{l,l}(N,\C)\cap
H^{2l}(N,\R)$ consists of vectors in $H^{l,l}(N,\C)$, invariant
under the complex conjugation, and that $H^{l,l}(N,\C) =
H^{l,l}(N,\R)\otimes_{\R}\C$. Let $h \in H^{1,1}(N,\R)^A$ be an
$A$-invariant K\"ahler class (e.g. any $A$-averaged K\"ahler
class). Note that the class of $h^l$ is a non-zero element in
$H^{l,l}(N,\R)^A$, $0\leq l\leq n$, and the Hodge -- Lefschetz
decomposition $H^k(N,\Q)\otimes_\Q\C = \oplus_{p+q+2l=k}
H^{p,q}_{prim}(N) \wedge h^l$ is $A$-invariant. Then from the
assumption $\dim H^{k,k}(N)^A = 1$ we get $H^{l,l}(N,\R)^A = \R
h^l$ for all $l\leq k$.

\def\widebar{\overline}

Further, since $A$ is abelian, any of its complex representations
$H^{p,q}_{prim}(N,\C)$ has a non-zero eigenvector $\alpha$ iff
$H^{p,q}_{prim}\neq0$. In this case, there is a character $\chi :
A \to \C^*$ such that $a^*(\alpha) = \chi(a)\cdot\alpha$ for all
$a \in A$, and moreover $\widebar{\chi(a)}=\chi(a)^{-1}$ because
$A$ is finite. Now, as the class $h$ is real, the
$A$-representations $H^{p+q}_{prim}(N)$ are defined over $\R$, and
we have $a^*(\bar\alpha) = \widebar{\chi(a)}\cdot\bar\alpha$. This
implies that $\beta := i \alpha\wedge\bar\alpha$ belongs to
$H^{p+q,p+q}(N,\R)^A$. Hence we get $\beta = C\cdot h^{p+q}$ for
some $C\in\R$. But then $C^2 h^{2(p+q)} = \beta^2 = i
\alpha\wedge\bar\alpha \wedge i \alpha\wedge\bar\alpha =
(\alpha\wedge\alpha)\wedge(\bar\alpha\wedge\bar\alpha) = 0$ (the
latter equality follows from super-commutativity of the
wedge-product and because $p+q$ is odd). So $C=0$, $\beta = 0$,
and thus $\int i\alpha\wedge\bar\alpha\wedge h^{n-(p+q)} = 0$. The
latter vanishing and the Hodge -- Riemann relations imply that
$\alpha=0$, that is $H^k_{prim}(N)=0$, and hence $b_k(N) =
b_{k-2}(N)$. The rest follows by induction.
\end{proof}

\begin{example}
\label{example:betti} Consider the quotient $f : N \longrightarrow
N/A =: M$. The pullback $f^*$ and the trace $f_*$ give an
identification of $H^{2k}(N,\Q)^A$ with $H^{2k}(M,\Q)$. Thus the
assumption $\dim H^{k,k}(N)^A = 1$ translates into $\dim
H^{k,k}(M)=1$. The case $k=1$ can be applied when $M$ is a (fake)
projective plane or any other variety with $b_2=1$ --- it says in
particular that all \emph{smooth finite abelian} covers of $M$
(possibly ramified) have \emph{finite} first homology.
\end{example}

We do not know any straightforward way to relax the condition for
the group $A$ to be abelian. Note for instance that the
permutation group $\Sigma_{d+1}$, $d \ge 1$, acts on the Cartesian
power $E^d$ of an elliptic curve $E$, with quotient being the
projective space $\P^d$.

\section{Automorphic forms of canonical weight and rigid divisors}
\label{section:weight3}

Let $f: S' \lra S$ be the universal unramified abelian cover of a
fake projective plane $S$.
\begin{lemma}
\label{lemma:regularity} The surface $S'$ is regular (i.e. $\dim
H^1(S',\OO_{S'}) = 0$).
\end{lemma}
\begin{proof}
By Example~\ref{example:betti}, this is a particular case of
Lemma~\ref{lemma:betti} applied to $N = S'$, $A = H_1(S,\Z)$, with
$M = N/A = S$ and $H^{1,1}(S')^{A} = H^{1,1}(S) \simeq \R$.
\end{proof}

Let $\varepsilon$ be a non-trivial torsion line bundle
 on $S$.

\begin{theorem}
\label{theorem:weight3} The linear system
$|\omega_S\otimes\varepsilon|$ consists of a \emph{unique} divisor
$D_\varepsilon$. More precisely, we have
$H^1(S,\omega_S\otimes\varepsilon) =
H^2(S,\omega_S\otimes\varepsilon) = 0$ and
$H^0(S,\omega_S\otimes\varepsilon) = \C s_\varepsilon$, with
$(s_\varepsilon)_0 = D_\varepsilon$.
\end{theorem}
\begin{proof}
All abelian covers are constructed as relative spectra, $S' =
\Spec_S \,\left(\displaystyle\bigoplus_{\varepsilon\in\Pic^0
S}\varepsilon\right)$, so that $f_*\OO_{S'} =
\displaystyle\bigoplus_{\varepsilon\in\Pic^0 S}\varepsilon$. Then
Lemma~\ref{lemma:regularity} and Leray spectral sequence give
\[0 = H^1(S',\OO_{S'}) = H^1(S,f_*\OO_{S'}) =
\bigoplus_{\varepsilon\in\Pic^0 S} H^1(S,\varepsilon).\] Thus
$H^1(S,\varepsilon) = 0$ for all $\varepsilon\in\Pic^0 S$. This
vanishing combined with Serre duality
$H^k(S,\omega_S\otimes\varepsilon) = H^{2-k}(S,\varepsilon^*)^*$
yields
$$h^0(S,\omega_S\otimes\varepsilon)+h^0(S,\varepsilon^*) =
h^0(S,\omega_S\otimes\varepsilon)+h^0(S,\varepsilon^*)-h^1(S,\varepsilon^*)
= \chi(S,\varepsilon^*) = 1.$$ Finally, since
$h^0(S,\varepsilon^*) = 0$ for all torsion line bundles $\OO_S \ne
\varepsilon\in\Pic^0 S$, we deduce that
$h^0(S,\omega_S\otimes\varepsilon) = 1$.
\end{proof}

This concludes the description of all spaces of modular forms of
canonical weight. The rest of this section gathers some immediate
corollaries from the uniqueness of divisors of canonical weight.

\begin{corollary}
Define $\Theta^1_+ := \{L\in\Pic^1S \ \big\vert \ h^0(L)\neq0\}$
for $\Pic^1S := \{L\in\Pic S \ \big\vert \ c_1(L) = c_1(\OO_S(1))
= 1\}$. Then we have:
\begin{enumerate}
\item All divisors of canonical weight are irreducible $\iff$
$\Theta^1_+ = \emptyset$.
\item $\binom{|\Theta^1_+|+2}{3}\leq|H_1(S,\Z)|-1$.
\end{enumerate}
\end{corollary}

\begin{corollary}
For all line bundles $L\in\Pic S$,
$\OO_S\neq\varepsilon\in\Pic^0S$, tensor multiplication by
$s_\varepsilon$ gives a monomorphism $\otimes s_\varepsilon:
H^0(S,L) \to H^0(S,L\otimes\omega_S\otimes\varepsilon)$. Also, if
both $L$ and $\varepsilon$ happen to be $A_S$-invariant, then
$\otimes s_\varepsilon$ is a monomorphism of $G$-representations
(cf. Section {\ref{subsection:int-3}}).
\end{corollary}

\section{Tautological bundles, liftings, linearizations, central extension and localization}
\label{section:pre}

We begin by recalling the next

\begin{lemma}[{see \cite[Lemma 2.1]{GKMS}}]
\label{lemma:Kdiv} Let $S$ be a fake projective plane with no
$3$-torsion in $H_1(S,\Z)$. Then there exists a \emph{unique}
(ample) line bundle $\OO_S(1)$ such that $\omega_S\cong\OO_S(3)$.
\end{lemma}

\refstepcounter{equation}
\subsection{Linearizations of line bundles}
\label{subsection:pre-1}

Let $S$ and $\OO_S(1)$ be as in Lemma~\ref{lemma:Kdiv}. We will
assume in what follows that $A_S = G_{21}$ or $(\Z\slash 3)^2$.
This implies that one has a lifting of the fundamental group
$\Pi\subset PU(2,1)$ to $SU(2,1)$ (see Table A above). Fix such a
lifting $r: \Pi \hookrightarrow SU(2,1)$ and consider the central
extension
\begin{equation}\label{cen-ext}
    1 \to \Z/3 \to SU(2,1) \to PU(2,1) \to 1.
\end{equation}
Note that the embedding $r$ is \emph{unique} because $H_1(S,\Z) =
\Pi\slash[\Pi,\Pi]$ does not have $3$-torsion in our case.

So, we get a \emph{linear} action of $r(\Pi)$ on $\C^3$. In
particular, both $\text{Bl}_0\,\C^3 = \text{Tot}\,\OO_{\P^2}(-1)$
and its restriction to the ball $\B\subset\P^2$ are preserved by
$r(\Pi)$, so that we get the equality
\[\text{Tot}\,\OO_S(1) =
(\text{Tot}\,\OO_{\P^2}(-1)\big\vert_{\B})\slash r(\Pi)\] (cf.
\cite[8.9]{Kollar95}). Further, since there is a natural
identification
\[ \Pic^0S = \Hom(H_1(S,\Z),\C^*) = \Hom(\Pi,\C^*), \]
every torsion line bundle $\varepsilon\in\Pic^0S$ corresponds to a
character $\chi_{\varepsilon}: \Pi \to \C^*$. One may twist the
fiberwise $r(\Pi)$-action on $\text{Tot}\,\OO_{\P^2}(-1)$ by
$\chi_{\varepsilon}$ (we will refer to this modified action as
$r(\Pi)_{\chi_{\varepsilon}}$) and obtain
\[\text{Tot}\,\OO_S(1)\otimes\varepsilon = (\text{Tot}\,\OO_{\P^2}(-1)\big\vert_{\B})\slash
r(\Pi)_{\chi_{\varepsilon}}.\footnote{~All these matters are
treated extensively in \cite[\text{Ch}. 8]{Kollar95}}\]

Observe that according to Table A there always exists such
$\varepsilon\ne 0$. This table also shows that in two cases one
can choose $\varepsilon$ to be $A_S$-invariant.

\refstepcounter{equation}
\subsection{The central extension}
\label{subsection:pre-1-a}

Note that $A_S = N(\Pi)\slash\Pi$ for the normalizer $N(\Pi)
\subset PU(2,1)$ of $\Pi$. Then \eqref{cen-ext} yields a central
extension
\[1 \to \Z/3 \to G \to A_S \to 1\]
for $G := \widetilde{N(\Pi)}\slash r(\Pi)$ and the preimage
$\widetilde{N(\Pi)}\subset SU(2,1)$ of $N(\Pi)$. Further, the
preceding construction of $\OO_S(1)$ is
$\widetilde{N(\Pi)}$-equivariant by the $A_S$-invariance of
$\OO_S(1)$, which gives a \emph{linear} $G$-action on all the
spaces $H^0(S,\OO_S(k)),k\in\Z$. Similarly, if the torsion line
bundle $\varepsilon$ is $A_S$-invariant, we get a linear
$G$-action on all $H^0(S,\OO_S(k)\otimes\varepsilon)$.

Recall next that when $A_S = G_{21}$, the line bundle $\OO_S(1)$
is $A_S$-linearizable (i.e. the group $A_S$ lifts to $G$ and the
corresponding extension splits), and so the spaces
$H^0(S,\OO_S(k))$ are some linear $A_S$-representations in this
case (see Table A or \cite[Lemma 2.2]{GKMS}). The same holds for
all $H^0(S,\OO_S(k)\otimes\varepsilon)$ and any $A_S$-invariant
$\varepsilon$.

In turn, if $A_S = (\Z\slash 3)^2$, then the extension $G$ of this
$A_S$ does not split (see Table A), i.e. $G$ coincides with the
\emph{Heisenberg group} $H_3$ of order $27$. Again the $G$-action
on $H^0(S,\OO_S(k))$ (resp. on
$H^0(S,\OO_S(k)\otimes\varepsilon)$) is linear here.

\begin{remark}
\label{remark-H-3} Let $\xi,\eta\in G = H_3$ be two elements that
map to order $3$ generators of $A_S = (\Z\slash 3)^2$. Then their
commutator $[\xi,\eta]$ generates the center $\Z/3 \subset G$, and
one obtains the following irreducible $3$-dimensional
(Schr\"odinger) representation of $G$:
\[\xi: x_i \mapsto \omega^{-i} x_i, \qquad \eta: x_i \mapsto x_{i+1}\quad (i\in\Z\slash 3,\ \omega := e^{\frac{2\pi\sqrt{-1}}{3}}),\]
with $x_i$ forming a basis in $\C^3$ (this representation,
together with its complex conjugate, are the only $3$-dimensional
irreducible representations of $H_3$). Observe at this point that
according to the above discussion $[\xi,\eta]$ acts
\emph{trivially} on $S$ and via the fiberwise scaling by $\omega$
on $\OO_S(1)$. Furthermore, given any $A_S$-invariant
$\varepsilon\in\Pic^0S$, the commutator $[\xi,\eta]$ acts
\emph{non-trivially} on $\OO_S(k)\otimes\varepsilon$ whenever $k$
is coprime to $3$.
\end{remark}

\refstepcounter{equation}
\subsection{Localization}
\label{theorem:hlf}

Let $S$ be a compact complex manifold. For a holomorphic
endomorphism $\tau: S\longrightarrow S$ and a linearized coherent
sheaf $F\in\Coh^\tau S$ (i.e. a coherent sheaf $F\in\Coh S$
equipped with a morphism $\phi:\tau^*F\to F$) consider the
compositions $H^k(\phi)\circ\tau^*: H^k(S,F) \to H^k(S,F)$ of
pullbacks $\tau^*:H^k(S,F)\to H^k(S,\tau^*F)$ with
$H^k(\phi):H^k(S,\tau^*F)\to H^k(S,F)$. Define the \emph{Lefschetz
number}
\[ L(\tau,F,\phi) := \sum_k (-1)^k \Tr(H^k(\phi)\circ\tau^* : H^k(S,F)). \]

\begin{theorem}[Woods Hole formula, see {\cite[Theorem 2]{Atiyah-Bott}}]
\label{theorem:woodshole} If the graphs of $\tau$ and $Id_S$ (the
diagonal) intersect transversally in $S\times S$, then
\[ L(\tau,F,\phi) =
\sum_{f(P)=P} \frac{\Tr(\tau(P) : F(P))}{\det(1-d\tau(P),
\Omega_S(P))},
\]
where $d\tau: \tau^* \Omega_S \to \Omega_S$ is the differential
(canonical linearization of the cotangent bundle).
\end{theorem}
Explicitly, for $\dim S = 2$ we get: $\det(1-d\tau(P)) =
(1-\alpha_1(P))(1-\alpha_2(P))$, where $\alpha_j(P)$ are the
eigenvalues of the Jacobian matrix $J_\tau$.

\refstepcounter{equation}
\subsection{}
\label{subsection:pre-3}

Let us conclude this section by recalling well-known supplementary
results.

It follows from the Riemann -- Roch formula that
\[\chi(\OO_S(k)\otimes\varepsilon) = \frac{(k-1)(k-2)}{2}\]
for all $k\in\Z$ and $\varepsilon\in\Pic^0S$.


In particular, we get $H^0(S,\OO_S(k)\otimes\varepsilon) \simeq
\C^{(k-1)(k-2)/2}$ when $k\geq4$, since
\[H^i(S,\OO_S(k)\otimes\varepsilon) = H^i(S,\omega_S\otimes\OO_S(k - 3)\otimes\varepsilon) = 0\]
for all $i > 0$ by Kodaira vanishing.

\section{Automorphic forms of large weight: computation of characters}
\label{section:weight-large}

Let $\tau \in A_S$ be any element of order $3$ with a
\emph{faithful} action on $S$. Choose any linearization $\phi:
\tau^*\OO_S(1)\to\OO_S(1)$ (cf. Section {\ref{subsection:pre-1}}).

\begin{lemma}
\label{lemma:all-weights-differ} The automorphism $\tau$ has only
three fixed points, in the Woods Hole formula (see
Theoorem~\ref{theorem:woodshole}) all denominators coincide (and
are equal to $3$), and all numerators are three \emph{distinct}
$3$rd roots of unity. In other words, one can choose a numbering
of the fixed points $P_i$ and weights $\alpha_j(P_i)$ in such a
way that the following holds:
\begin{itemize}
\item $\alpha_j(P_i) = \omega^j$ for all $i,j$;
\item $w_{i} := \Tr(\tau(P_i) : \OO_S(1)_{P_i}) = \omega^{i}$ for all $i$.
\end{itemize}
Moreover, for any $\tau$-invariant line bundle
$\varepsilon\in\Pic^0S$ of order $m$ coprime to $3$, we can choose
a linearization such that $\Tr\tau|_{\varepsilon_{P_i}} = 1$. In
particular, $\Tr\tau|_{(\OO_S(k)\otimes\varepsilon^{l})_{P_i}} =
\omega^{ik}$ for all $i,k,l$. (We are using the notation
$\varepsilon^i := \varepsilon^{\otimes i}$.)
\end{lemma}

\begin{proof}[Proof]
The claim about $P_i$ and $\alpha_j(P_i)$ follows from
\cite[Proposition 3.1]{Keum08}.

We have $V := H^0(S,\OO_S(4)) \cong\C^3$ (see Section
{\ref{subsection:pre-3}}). Let $v_1, v_2, v_3$ be the eigen values
of $\tau$ acting on $V$. Then $v_i^3 = 1$ for all $i$ and
$\Tr\tau\big\vert_V = v_1 + v_2 + v_3$. At the same time, as
follows from Section {\ref{theorem:hlf}}, we have
\[\Tr\tau\big\vert_V = \frac{w_1^4 + w_2^4 + w_3^4}{(1-\omega)(1-\omega^2)} = \frac{w_1 + w_2 + w_3}{3}.\]
The latter can be equal to $v_1 + v_2 + v _3$ only when all $v_i$
(resp. all $w_i$) are pairwise distinct (so that both sums are
zero). This is due to the fact that the sum of three $3$rd roots
of unity has the norm $\in \{0,\sqrt{3},3\}$ and is zero iff all
roots are distinct.

Finally, since $(m,3) = 1$, we may replace $\varepsilon$ by
$\varepsilon^{3}$, so that the action of $\tau$ on the closed
fibers $\varepsilon_{P_i}$ is trivial. The last assertion about
$\OO_S(k)\otimes\varepsilon^{l}$ is evident.
\end{proof}

Fix some $\varepsilon\in(\Pic^0 S)^{A_S}$. Recall that the group
$G$ from Section {\ref{subsection:pre-1-a}} acts linearly on all
spaces $V := H^0(S,\OO_S(k)\otimes\varepsilon)$.

\begin{proposition}
\label{tr-is-inv} Let $G = H_3$. Then for $k\geq4$ the following
holds (cf. Table B below):
\begin{itemize}
\item for $k \equiv 0 \mod 3$, we have
    $V = V_0 \oplus \C[(\Z/3)^2]^a$ as $G$-representations, where
    $a := \displaystyle\frac{k}{6}\,(\displaystyle\frac{k}{3}-1)$
    and $V_0 \simeq \C$ (resp. $\C[(\Z\slash 3)^2]$) is the trivial
    (resp. regular) representation;
\item for $k \equiv 1 \mod 3$, we have $V = \VV_3^{\oplus
    (k-1)(k-2)/6}$ as $G$-representations, where $\VV_3$
    is an irreducible $3$-dimensional representation of $H_3$;
\item for $k \equiv 2 \mod 3$, we have $V = \VVo_3^{\oplus
    (k-1)(k-2)/6}$ as $G$-representations, where $\VVo_3$
    is the complex conjugate to $\VV_3$ above.
\end{itemize}
\end{proposition}

\begin{table}[h]
\begin{longtable}{|c|c|c|c|}
\hline
$4 \le k \mod 3$ & $0$ & $1$ & $2$ \\
\hline $H^0(S,\OO_S(k)\otimes\varepsilon)$ & $V_0
\oplus\C[(\Z/3)^2]^a$ & $\VV_3^{\oplus(k-1)(k-2)/6}$ &
$\VVo_3^{\oplus (k-1)(k-2)/6}$ \\
\hline
\end{longtable}
\text{Table B}
\end{table}

\begin{proof}[Proof]
Suppose that $k \equiv 0 \mod 3$. Then, since every element in
$H_3$ has order $3$, applying Lemma~\ref{lemma:all-weights-differ}
to any non-central $\tau \in G$ we obtain $(1 - \alpha_1(P_i))(1 -
\alpha_2(P_i)) = 3$ for all $i$ and $\Tr\tau\big|_{V} = 1$ (cf.
Section {\ref{theorem:hlf}}).

Further, the element $[\xi,\eta] \in G$ from
Remark~\ref{remark-H-3} acts trivially on $\OO_S(k)$ (via scaling
by $\omega^k = 1$). Also, since the order of any $\varepsilon \ne
0$ is coprime to $3$ (see Table A) and $\varepsilon$ is flat (see
its construction in Section {\ref{subsection:pre-1}}), from
Lemma~\ref{lemma:all-weights-differ} we find that $[\xi,\eta]$
acts trivially on $\OO_S(k)\otimes\varepsilon$, hence on $V$ as
well. This implies that the $G$-action on $V$ factors through that
of its quotient $(\Z/3)^2$. Then the claimed decomposition $V =
V_0 \oplus \C[(\Z/3)^2]^a$ follows from the fact that $1 + 9a =
\dim V = \Tr[\xi,\eta]\big|_{V}$ and that $\Tr\tau\big|_{V} = 1$
for all non-central $\tau \in G$.

Let now $k \equiv 1 \mod 3$ (resp. $k \equiv 2 \mod 3$). Then it
follows from Remark~\ref{remark-H-3} that $[\xi,\eta]$ scales all
vectors in $V = H^0(S,\OO_S(k)\otimes\varepsilon)
\cong\C^{(k-1)(k-2)/2}$ by $\omega^k \ne 1$. Furthermore, since
$\Tr\xi\big|_{V} = 0 = \Tr\eta\big|_{V}$ by
Lemma~\ref{lemma:all-weights-differ} and Section
{\ref{theorem:hlf}}, all irreducible summands of $V$ are faithful
$G$-representations, and hence they are isomorphic to $\VV_3$
(resp. to $\VVo_3$). This concludes the proof.
\end{proof}

\begin{proposition}
\label{tr-is-inv-21} Let $G \supset A_S =  G_{21}$. Then for
$k\geq4$ and $V = H^0(S,\OO_S(k)\otimes\varepsilon)$, we have the
following equality of (virtual) $G$-representations
\[
V = \C[G]^{\oplus a_k} \oplus U_k
\]
for some $a_k \in \Z$ expressed in terms of $\dim V$, where $U_k$
depends only on $k \mod 21$ and is explicitly given in the table
below, with rows (resp. columns) being enumerated by $k \mod 3$
(resp. $k \mod 7$)

\begin{center}

\begin{tabular}{|c|c|c|c|c|c|c|c|}
\hline
& $0$ & $1$ & $2$ & $3$ & $4$ & $5$ & $6$ \\

\hline $0$ &    $\C$              & $\VV_3 \oplus \VVo_3 \oplus
\C$ & $\VV_3 \oplus \VVo_3 \oplus \C$ & $\C$ & $\VV_3^{\oplus 2}
\oplus \VVo_3 \oplus \C$ & $(\VV_3 \oplus \VVo_3)^{\oplus 2}
\oplus \C$ & $\VV_3 \oplus
\VVo_3^{\oplus 2} \oplus \C$ \\
\hline $1 \ \text{or} \ 2$  & $(-\VV_3) \oplus (-\VVo_3)$ & $0$ &
$0$ & $(-\VV_3) \oplus (-\VVo_3)$ & $\VVo_3$ & $\VV_3 \oplus
\VVo_3$ & $\VV_3$ \\
\hline
\end{tabular}
\end{center}
\end{proposition}

\begin{proof}[Proof]
From Section {\ref{theorem:hlf}} we obtain
\[\Tr\sigma\big|_V = \frac{\zeta^{6k}}{(1-\zeta)(1-\zeta^3)} + \frac{\zeta^{5k}}{(1-\zeta^2)(1-\zeta^6)} + \frac{\zeta^{3k}}{(1-\zeta^4)(1-\zeta^5)}.\]
Here $\zeta := e^{\frac{2\pi\sqrt{-1}}{7}}$ and $\sigma\in G_{21}$
is an element of order $7$. The value $\Tr\sigma\big|_V$ depends
only on $k \mod 7$ and by the direct computation we obtain the
following table:

\begin{longtable}{|c|c|c|c|c|c|c|c|c|}
\hline
$k$ & $0$ & $1$ & $2$ & $3$ & $4$ & $5$ & $6$ \\
\hline $\Tr\sigma\big|_V$ & $1$ & $0$ & $0$ & $1$ & $\bar{b}$ &
$-1$ &
$b$ \\
\hline
\end{longtable}
(Here $b := \zeta + \zeta^2 + \zeta^4$ and $\bar{b} = -1 - b =
\zeta^3 + \zeta^5 + \zeta^6$ are Gauss sums
$\frac{-1\pm\sqrt{-7}}2$.)

Let $\tau \in G_{21}$ be an element of order $3$ such that $G_{21}
= \left<\sigma,\tau\right>$. Recall the character table for the
group $G_{21}$ (see e.g. the proof of \cite[Lemma 4.2]{GKMS}):


\begin{center}
\begin{tabular}{|c|c|c|c|c|c|}
\hline
& $1$ & $\Tr\sigma$ & $\Tr\sigma^3$ & $\Tr\tau$ & $\Tr\tau^2$ \\
\hline
$\C$ &                    $1$ & $1$ & $1$ & $1$ & $1$ \\
\hline
$\VV_1$  &                $1$ & $1$ & $1$ & $\omega$ & $\overline{\omega}$ \\
\hline
$\VVo_1$ &    $1$ & $1$ & $1$ & $\overline{\omega}$ & $\omega$ \\
\hline
$\VV_3$ &                 $3$ & $b$ & $\overline{b}$ & $0$ & $0$ \\
\hline
 $\VVo_3$ &  $3$ & $\overline{b}$ & $b$ & $0$ & $0$ \\
\hline
\end{tabular}
\end{center}
Here ``---'' signifies, as usual, the complex conjugation and
$\VV_i$ are irreducible $i$-dimensional representations of
$G_{21}$.

Now from Lemma~\ref{lemma:all-weights-differ} and the preceding
tables we get the claimed options for $V$. This concludes the
proof.
\end{proof}

\section{Automorphic forms of small weight: two proofs of vanishing}
\label{section:weight2}

\begin{lemma}[{see \cite[Lemma 15.\,6.\,2]{Kollar95}}]
\label{lemma:kollar} Let $M_i\subset H^0(X,L_i)$ be non-zero
linear systems on a proper normal variety $X$. Then the image of
the natural linear map $p: \bigotimes_i M_i \to H^0(X,\bigotimes_i
L_i)$ has dimension at least $1 + \sum_i (\dim M_i-1)$. In
particular, $\dim M_i\leq\frac{h^0(X,L_i^{\otimes 2})+1}{2}$.
\end{lemma}
\begin{proof}
A divisor $D\in\P H^0(X,\bigotimes_i L_i)$ has bounded number of
irreducible components, so it can be written as a sum of divisors
from $\P M_i$ in finitely many ways, thus the natural map from the
Segre variety $\prod_i \P M_1$ to $\P H^0(X,\bigotimes_i L_i)$ has
finite fibers, which implies that the dimension of the
projectivization of the image of $p$ is not less than $\dim
\prod_i \P(M_i) = \sum_i (\dim M_i - 1)$.
\end{proof}

\begin{lemma}
\label{lemma:remake} Let $L\in\Pic^G X$ be a $G$-linearized
holomorphic line bundle on a compact complex variety $X$ equipped
with an action of a finite group $G:X$ (not necessarily faithful).
Assume that $H^0(X,L^{\otimes 2})$ is a $G$-representation of
dimension at most $4$ that does not contain $1$-dimensional
subrepresentations. If a vector space $H^0(X,L)$ is non-zero, then
it is an irreducible $2$-dimensional $G$-representation and the
order of $G$ is even.
\end{lemma}
\begin{proof}
Lemma~\ref{lemma:kollar} implies that $h^0(X,L) \leq 2$. Thus we
must exclude the case when $H^0(X,L)$ contains a $1$-dimensional
subrepresentation $M$. But then by Lemma~\ref{lemma:kollar} a
$1$-dimensional representation $M^{\otimes 2}$ embeds into
$H^0(X,L^{\otimes 2})$, contradiction. Finally observe that
irreducible representations of odd order groups have odd degree.
\end{proof}

We retain the earlier notation.
\begin{theorem}
\label{theorem:weight2} Let $S$ be a fake projective plane with
the automorphism group $A_S$ of order at least $9$ and
$\varepsilon \in (\Pic S)^{A_S}$ be an $A_S$-invariant torsion
line bundle. Then $H^0(S,\OO_S(2)\otimes\varepsilon) = 0$.
\end{theorem}

\begin{proof}[The old proof]
\label{old_proof} Propositions~\ref{tr-is-inv} and
\ref{tr-is-inv-21} imply that $H^0(S,\OO_S(4)\otimes\varepsilon)$
is an irreducible $3$-dimensional $G$-representation.
Lemma~\ref{lemma:remake} implies that
$H^0(S,\OO_S(2)\otimes\varepsilon) = 0$, exactly as in the proof
of \cite[Theorem 1.3]{GKMS}.
\end{proof}

\emph{The new proof} comes with the extra assumption
$\varepsilon\ne 0$ (cf. the end of Section
{\ref{subsection:pre-1}}).

Suppose that $H^0(S,\OO_S(2) \otimes \varepsilon)\ne0$. There is a
natural homomorphism of $G$-modules
\[H^0(S,\omega_S\otimes\varepsilon^*) \otimes H^0(S,\OO_S(2)\otimes
\varepsilon) \to H^0(S,\OO_S(5))\] and
Theorem~\ref{theorem:weight3} implies that $H^0(S,\OO_S(2)\otimes
\varepsilon)$ is a \emph{non-trivial} subrepresentation in
$H^0(S,\OO_S(5))$ of dimension $\le 2$. But the latter contradicts
Proposition~\ref{tr-is-inv} (for $A_S=(\Z\slash3)^2$) or
Proposition~\ref{tr-is-inv-21} (for $A_S = G_{21}$). Thus
Theorem~\ref{theorem:weight2} follows in this case.

Similarly, the $G$-homomorphism
\[H^0(S,\omega_S\otimes\varepsilon) \otimes H^0(S,\OO_S(2)) \to H^0(S,\OO_S(5)\otimes\varepsilon)\]
gives contradiction with either Proposition~\ref{tr-is-inv} or
Proposition~\ref{tr-is-inv-21}, provided that $H^0(S,\OO_S(2))\ne
0$. This concludes the new proof. \qed

\bibliographystyle{amsplain}

\begin{thebibliography}{41}

\bibitem{ABKW}
Klaus Altmann, Jaros{\l}aw Buczy{\'n}ski, Lars Kastner, Anna-Lena
Winz: \emph{Immaculate line bundles on toric varieties}, Pure and
Applied Mathematics Quarterly, \textbf{16}:4 (2020), 1147 -- 1217.

\bibitem{Atiyah-Bott}
Michael Atiyah, Raoul Bott: \emph{A Lefschetz fixed point formula
for elliptic differential operators}, Bull. Amer. Math. Soc.,
\textbf{72} (1966), 245 -- 250.

\bibitem{Barlow}
Rebecca Barlow: \emph{Zero-cycles on Mumford's surface},
Mathematical Proceedings of the Cambridge Philosophical Society,
\textbf{126} (1999), no. 3, 505 -- 510.
\doi{10.1017/S0305004198003442}

\bibitem{Beilinson78}
Alexander Beilinson: \emph{Coherent sheaves on $\P^n$ and problems
in linear algebra}, Funktsional. Anal. i Prilozhen., \textbf{12}
(1978), no. 3, 68 -- 69; Funct. Anal. Appl., \textbf{12} (1978),
no. 3, 214 -- 216.

\bibitem{Blasius-Rogawski}
Don Blasius, Jonathan Rogawski: \emph{Cohomology of congruence
subgroups of $SU(2,1)^p$ and Hodge cycles on some special complex
hyperbolic surfaces}, In: Regulators in Analysis, Geometry and
Number Theory (Eds. A. Resnikov and N. Schappacher), Progress In
Mathematics 171, Birkhauser 1999.

\bibitem{Bloch80}
Spencer Bloch: \emph{Lectures on algebraic cycles}, Duke
University Mathematics Series. IV (1980). Durham, North Carolina:
Duke University, Mathematics Department.

\bibitem{Borisov-Fatighenti}
Lev Borisov, Enrico Fatighenti: \emph{New explicit constructions
of surfaces of general type}, \arxiv{2004.02637}.

\bibitem{BBF}
Lev Borisov, Anders Buch, Enrico Fatighenti: \emph{A journey from
the octonionic $\P^2$ to a fake $\P^2$}, \arxiv{2008.09731}.

\bibitem{Borisov-Keum}
Lev Borisov, Jonghae Keum: \emph{Explicit equations of a fake
projective plane}, Duke Math. J.  \textbf{169} (2020), no. 6, 1135
-- 1162. \doi{10.1215/00127094-2019-0076}

\bibitem{Brino-Cerbo}
Gennaro Di Brino, Luca Di Cerbo: \emph{Exceptional collections and
the bicanonical map of Keum's fake projective planes}, Commun.
Contemp. Math. \textbf{20} (2018), no. 1, 1650066, 13 pp. Also
\arxiv{1603.04378}.

\bibitem{Cartwright-Steger}
Donald Cartwright, Tim Steger: \emph{Enumeration of the $50$ fake
projective planes}, C. R. Acad. Sci. Paris, Ser. I, \textbf{348}
(2010), 11 -- 13. \doi{10.1016/j.crma.2009.11.016}

\bibitem{CScode}
Donald Cartwright, Tim Steger:
\url{http://www.maths.usyd.edu.au/u/donaldc/fakeprojectiveplanes/}

\bibitem{Fakhruddin15}
Najmuddin Fakhruddin: \emph{Exceptional collections on
$2$-adically uniformised fake projective planes}, Mathematical
Research Letters, \textbf{22} (2015), no. 1, 43 -- 57. Also
\arxiv{1310.3020}.

\bibitem{Galkin-Iritani}
Sergey Galkin, Hiroshi Iritani: \emph{Gamma conjecture via mirror
symmetry}, Advanced Studies in Pure Mathematics: Primitive forms
and related subjects, \textbf{83} (2019), 53 -- 115.
\doi{10.2969/aspm/08310055} Also \arxiv{1508.00719}.

\bibitem{GKMS}
Sergey Galkin, Ludmil Katzarkov, Anton Mellit, Evgeny Shinder:
\emph{Derived categories of Keum's fake projective planes},
Advances in Mathematics, \textbf{278} (2015), 238 -- 253.
\doi{10.1016/j.aim.2015.03.001}

\bibitem{GKMS-mp}
Sergey Galkin, Ludmil Katzarkov, Anton Mellit, Evgeny Shinder:
\emph{Minifolds and phantoms}, \arxiv{1305.4549}.

\bibitem{HMT}
Claus Hertling, Yuri Manin, Constantin Teleman: \emph{An update on
semisimple quantum cohomology and F-manifolds}, Proceedings of the
Steklov Institute of Mathematics, \textbf{264} (2009), no. 1, 62
-- 69. Also \arxiv{0803.2769}.

\bibitem{Hirzebruch58}
Friedrich Hirzebruch: \emph{Automorphe {F}ormen und der {S}atz von
{R}iemann -- {R}och}, Symposium internacional de topolog\'{\i}a
algebraica ({I}nternational symposium on algebraic topology):
Universidad Nacional Aut\'{o}noma de M\'{e}xico and UNESCO, Mexico
City, (1958), 129 -- 144.

\bibitem{Hirzebruch66}
Friedrich Hirzebruch: \emph{Elliptische Differentialoperatoren auf
Mannigfaltigkeiten}, Festschr. Ged\"achtnisfeier K. Weierstrass,
Westdeutscher Verlag, Cologne, (1966), 583 -- 608.
Uspehi Mat. Nauk, \textbf{23} (1968), no. 1 (139), 191 -- 209.

\bibitem{Ise64}
Mikio Ise: \emph{Generalized automorphic forms and certain
holomorphic vector bundles}, Amer. J. Math., \textbf{86} (1964),
no. 1, 70 -- 108. \doi{10.2307/2373036}

\bibitem{Ishida88}
Masa-Nori Ishida: \emph{An elliptic surface covered by Mumford's
fake projective plane}, Tohoku Math. J., \textbf{40} (1988), no.
3, 367 -- 396.

\bibitem{Kato-Ochiai}
Fumiharu Kato, Hiroyuki Ochiai: \emph{Arithmetic structure of CMSZ
fake projective planes}, J. Algebra, \textbf{305} (2006), no. 2,
1166 -- 1185. Also \arxiv{math/0006223}.

\bibitem{Kazhdan67}
David A. Kazhdan: \emph{Connection of the dual space of a group
with the structure of its close subgroups}, Funktsional. Anal. i
Prilozhen., \textbf{1} (1967), no. 1, 71 -- 74; Funct. Anal.
Appl., \textbf{1} (1967), no. 1, 63 -- 65.

\bibitem{Keum08}
Jonghae Keum: \emph{Quotients of fake projective planes}, Geom.
Topol., \textbf{12} (2008), no. 4, 2497 -- 2515. Also
\arxiv{0802.3435}.

\bibitem{Keum17}
Jonghae Keum: \emph{A vanishing theorem on fake projective planes
with enough automorphisms}, Trans. Amer. Math. Soc., \textbf{369}
(2017), no. 10, 7067 -- 7083. Also \arxiv{1407.7632v3}.

\bibitem{Kharlamov-Kulikov14}
Vyacheslav Kharlamov, Viktor Kulikov: \emph{On numerically
pluricanonical cyclic coverings}, Izv. RAN. Ser. Mat., \textbf{78}
(2014), no. 5, 143 -- 166; Izv. Math., \textbf{78} (2014), no. 5,
986 -- 1005. Also \arxiv{1308.0516}.

\bibitem{Klingler03}
Bruno Klingler: \emph{Sur la rigidit\'e de certains groupes
fonndamentaux, l'arithm\'eticit\'e des r\'eseaux hyperboliques
complexes, et les `faux plans projectifs'}, Invent. Math.,
\textbf{153} (2003), 105 -- 143. \doi{10.1007/s00222-002-0283-2}

\bibitem{Kobayashi-Ono}
Toshiyuki Kobayashi, Kaoru Ono: \emph{Note on Hirzebruch's
proportionality principle}, Jour. Fac. Sci. Univ. Tokyo,
\textbf{37} (1990), no. 1, 71 -- 87.

\bibitem{Kollar95}
J\'anos Koll\'ar: \emph{Shafarevich maps and automorphic forms},
M. B. Porter Lectures, Princeton Univ. Press, Princeton, NJ, 1995.

\bibitem{Lai-Yeung}
Ching-Jui Lai, Sai-Kee Yeung: \emph{Exceptional collection of
objects on some fake projective planes}, available at
\url{http://www.math.purdue.edu/~yeung/papers/FPP-van.pdf}.

\bibitem{Langlands63}
Robert Langlands: \emph{The Dimension of Spaces of Automorphic
Forms}, Amer. J. Math., \textbf{85} (1963), no. 1, 99 -- 125.
\doi{10.2307/2373189}

\bibitem{LeBrun04}
Claude LeBrun: \emph{Curvature functionals, optimal metrics, and
the differential topology of 4-manifolds}, Different faces of
geometry, 199 -- 256, Int. Math. Ser. (N.Y.), 3, Kluwer/Plenum,
New York, 2004. \doi{10.1007/b115003}

\bibitem{Mumford79}
David Mumford: \emph{An algebraic surface with $K$ ample,
$K^2=9$,$p_g=q=0$}, Amer. J. Math., \textbf{101} (1979), no. 1,
233 -- 244. \doi{10.2307/2373947}

\bibitem{Prasad-Yeung}
Gopal Prasad, Sai-Kee Yeung: \emph{Fake projective planes},
Invent. Math., \textbf{168} (2007), 321 -- 370; \textbf{182}
(2010), 213 -- 227. Also \arxiv{math/0512115v5},
\arxiv{math/0906.4932v3}.

\bibitem{Rogawski90}
Jonathan Rogawski: \emph{Automorphic Representations of Unitary
Groups in Three Variables}, Annals of Math. Studies: Princeton
University Press, \textbf{123} (1990).

\bibitem{Stover13}
Matthew Stover: \emph{ERRATUM AND ADDENDUM: ``PROPERTY (FA) AND
LATTICES IN SU(2,1)''}, International Journal of Algebra and
Computation, \textbf{23} (2013), no. 7, 1783 -- 1787.
\doi{10.1142/S0218196713920033}

\bibitem{Weissauer83}
Rained Weissauer: \emph{Vektorwertige Siegelsche Modulformen
kleinen Gewichtes}, J. Reine Angew. Math., \textbf{343} (1983),
184 -- 202.

\bibitem{Yau77}
Shing-Tung Yau: \emph{Calabi's conjecture and some new results in
algebraic geometry}, Proceedings of the National Academy of
Sciences, \textbf{74} (1977), no. 5, 1798 -- 1799.
\doi{10.1073/pnas.74.5.1798}

\bibitem{Yeung04}
Sai-Kee Yeung: \emph{Integrality and arithmeticity of co-compact
lattices corresponding to certain complex two-ball quotients of
Picard number one}, Asian J. Math., \textbf{8} (2004), 107 -- 130.

\end{thebibliography}

\newpage

\end{document}